\documentclass[12pt,a4paper,reqno]{amsart}
\usepackage[english]{babel}
\usepackage[applemac]{inputenc}
\usepackage[T1]{fontenc}
\usepackage{palatino}
\usepackage{amsmath}
\usepackage{amssymb}
\usepackage{amsthm}
\usepackage{amsfonts}
\usepackage{graphicx}
 \usepackage{amsrefs}
\usepackage[colorlinks = true, citecolor = black]{hyperref}
\title[Characterisation of BPLG using projections]{Characterising the big pieces of Lipschitz graphs property using projections}

\author{Henri Martikainen}
\address[H.M.]{Department of Mathematics and Statistics, University of Helsinki, P.O.B. 68, FI-00014 Helsinki, Finland}
\email{henri.martikainen@helsinki.fi}
\thanks{H.M. is supported by the Academy of Finland through the grant
Multiparameter dyadic harmonic analysis and probabilistic methods. T.O. is supported by the Academy of Finland through the grant Restricted families of projections and connections to
Kakeya type problems. Both authors are members of the Finnish Centre of Excellence in Analysis and Dynamics Research.}

\author{Tuomas Orponen}
\address[T.O.]{Department of Mathematics and Statistics, University of Helsinki, P.O.B. 68, FI-00014 Helsinki, Finland}
\email{tuomas.orponen@helsinki.fi}

\makeatletter
\@namedef{subjclassname@2010}{%
  \textup{2010} Mathematics Subject Classification}
\makeatother


\newcommand{\R}{\mathbb{R}}
\newcommand{\N}{\mathbb{N}}

\newcommand{\Z}{\mathbb{Z}}

\newcommand{\calH}{\mathcal{H}}

\newcommand{\spa}{\operatorname{span}}
\newcommand{\diam}{\operatorname{diam}}

\newcommand{\Se}{\mathbb{S}}

\numberwithin{equation}{section}

\theoremstyle{plain}
\newtheorem{thm}[equation]{Theorem}

\newtheorem{lemma}[equation]{Lemma}

\newtheorem{proposition}[equation]{Proposition}

\theoremstyle{definition}

\newtheorem{definition}[equation]{Definition}

\newtheorem{stopCondition}[equation]{Stopping condition}

\theoremstyle{remark}
\newtheorem{remark}[equation]{Remark}

\begin{document}
\begin{abstract} We characterise the \emph{big pieces of Lipschitz graphs} property in terms of projections. Roughly speaking, we prove that if a large subset of an $n$-Ahlfors-David regular set $E \subset \R^d$ has plenty of projections in $L^{2}$, then a large part of $E$ is contained in a single Lipschitz graph. This is closely related to a question of G. David and S. Semmes. 
\end{abstract}

\maketitle

\section{Introduction and statement of results}
The purpose of this paper is to characterise the \emph{big pieces of Lipschitz graphs} (BPLG) condition in terms of projections, and only projections. We begin with some definitions, and then formulate the characterisation. After that, we will discuss the context of the result and provide an outline of the proof.

We are concerned with $n$-Ahlfors-David regular sets in $\R^d$:
\begin{definition}[$n$-ADR]
Given $n \in \N$, a set $E \subset \R^d$ is $n$-Ahlfors-David regular ($n$-ADR) if $C_1r^n \le \calH^{n}(E \cap B(x,r)) \le C_2r^n$ for all points $x \in E$ and radii $r \in (0, \diam(E)]$, and some constants  $0 < C_{1} \leq C_{2} < \infty$. The constants $C_{1},C_{2}$ are referred to as the ADR constants of $E$.
\end{definition}
With an $n$-dimensional Lipschitz graph $\Gamma = \Gamma_A$ in $\R^d$ we mean a set of the form
\begin{displaymath}
\Gamma =  \{p + A(p)\colon\, p \in P\},
\end{displaymath}
where $P \subset \R^d$ is an $n$-dimensional subspace and $A\colon P \to P^{\perp}$ is a Lipschitz map. By the \emph{Lipschitz constant of $\Gamma$}, we mean the Lipschitz constant $\textup{Lip}(A)$ of $A$.
Our main object of study is the following subclass of $n$-ADR sets:
\begin{definition}[BPLG]
An $n$-ADR set $E \subset \R^{d}$ has \emph{big pieces of Lipschitz graphs}, if there exist constants $M < \infty$ and $\delta > 0$ with the following property: for every $x \in E$ and $r \in (0,\diam(E)]$, there exists an $n$-dimensional Lipschitz graph $\Gamma_{x,r}$ with Lipschitz constant at most $M$, such that $\calH^{n}(E \cap \Gamma_{x,r} \cap B(x,r)) \ge \delta r^n$.
\end{definition}
The closely related \emph{uniformly rectifiable} sets are defined as follows:
\begin{definition}[$n$-UR]\label{defn:UR}
A closed set $E \subset \R^d$ is $n$-uniformly rectifiable ($n$-UR) if it is $n$-ADR, and there exist constants $\delta > 0$ and $M < \infty$ with the following property:
for every $x \in E$ and $r \in (0, \diam(E)]$ there is a Lipschitz mapping $g\colon B_{\R^n}(0, r) \to \R^d$ such that Lip$(g) \le M$
and \begin{displaymath}
\calH^{n}(E \cap B(x,r) \cap g(B_{\R^n}(0,r)) ) \ge \delta r^n.
\end{displaymath} 
\end{definition}
For the basics of UR and BPLG sets, we refer to the monographs \cite{DS1, DS3} of David and Semmes. Notice that BPLG trivially implies UR. However, the converse is not true by an unpublished
example of T. Hrycak (see the discussion after Theorem \ref{maingen}). Let us also mention
the recent deep geometric result by J. Azzam and R. Schul \cite{AS}, which says that UR = (BP)$^2$LG, that is, UR sets contain big pieces of sets which have BPLG.

We need the following final definitions. For any $1 \le k < d$, denote by $G(d,k)$ the Grassmannian manifold of all $k$-dimensional subspaces of $\R^{d}$, equipped with the distance
$\|V - W\|_{G(d,k)} := \|\pi_V - \pi_W\|$. Here $\pi_V$ is the orthogonal projection onto $V$ and $\| \cdot \|$ is the usual operator norm of linear mappings.
There is also a natural Borel probability measure $\gamma_{d,k}$ on $G(d,k)$ -- see Chapter 3 of \cite{Mat}. Closed metric balls on $G(d,k)$ are usually denoted by $B(V,r)$, or $B_{G(d,k)}(V,r)$, if there is any risk of confusion. 

The next proposition gives an easy necessary condition for a set to have BPLG:
\begin{proposition}\label{LInftyProp}
Assume that $E \subset \R^{d}$ has BPLG. Then, there exist constants $\rho > 0$ and $C < \infty$ (depending only on $M$ and $\delta$ in the definition of BPLG) such that the following holds: for every $x \in E$, and every radius $r \in (0,\diam(E)]$, there exists a subspace $V_{x,r} \in G(d,n)$ and a subset $E_{x,r} \subset E \cap B(x,r)$ with the properties that $\calH^{n}(E_{x,r}) \ge r^n/C$, and
$\pi_{V\sharp}\calH^{n}|_{E_{x,r}} \in L^{\infty}(V)$ for every $V \in B(V_{x,r}, \rho)$ with the uniform bound
\begin{displaymath}
 \|\pi_{V\sharp}\calH^{n}|_{E_{x,r}}\|_{L^{\infty}(V)} \leq C.
 \end{displaymath}
\end{proposition}
\begin{proof}
Fix $x \in E$, $0 < r \leq \diam(E)$, and let $E_{x,r} = E \cap B(x,r) \cap \Gamma_{x,r}$, where $\Gamma_{x,r} = \{p + A(p)\colon\, p \in P\}$ is given by the BPLG condition, and Lip$(A) \le M$. Then, set $V_{x,r} = P$ and $\rho = [2(1+M)]^{-1}$. If $V \in B(V_{x,r}, \rho)$ one can also write
$\Gamma_{x,r} = \{v + A_V(v)\colon\, v \in V\}$, where $A_V\colon V \to V^{\perp}$ and Lip$(A_V) \le 3(1+M)$.

Suppose that $V$ is as above, $v \in V$ and $s > 0$. Then
\begin{displaymath}
\pi_V^{-1}(B_V(v,s)) \cap \Gamma_{x,r} \subset B(v+A_V(v), 4(1+M)s),
\end{displaymath}
which implies that 
\begin{displaymath}
\frac{ \pi_{V\sharp}\calH^{n}|_{E_{x,r}}(B_V(v,s))}{ \mathcal{H}^n(B_V(v,s))} \lesssim 1.
\end{displaymath}
The proposition now follows from Theorem 2.12 of \cite{Mat}.
\end{proof}
\begin{remark}
We write $A \lesssim_{p} B$ if $A \le CB$ for some constant $C > 0$ depending only on the parameter $p$; the notation $A \lesssim B$ means that the constant $C$ is absolute, or depends only on parameters, which can be regarded as "fixed" in the situation. The two-sided inequality $B \lesssim_{p} A \lesssim_{p} B$ is abbreviated to $A \sim_{p} B$.
\end{remark}

The main result of the paper asserts that the necessary condition for BPLG in Proposition \ref{LInftyProp} is also sufficient, and the uniform $L^{\infty}$-bound for the projections can even be relaxed to an averaged bound for the $L^{2}$-norms:
\begin{thm}\label{maingen}
Let $E \subset \R^d$ be an $n$-ADR set. Suppose that there exist constants $\kappa > 0$ and $C < \infty$ such that the following holds. For every $x \in E$ and $r \in (0, \diam(E)]$, there is an $\calH^{n}$-measurable subset $E_{x,r} \subset E \cap B(x,r)$ and a subspace $V_{x,r} \in G(d,n)$ with the following properties:
\begin{enumerate}
\item $\mathcal{H}^n(E_{x,r}) \ge \kappa r^n$,
\item $\pi_{V \sharp}\calH^{n}|_{E_{x,r}} \in L^2(V)$ for $\gamma_{d,n}$-a.e. $V \in B(V_{x,r}, \kappa)$, and
\begin{displaymath}
\int_{B(V_{x,r}, \kappa)} \|\pi_{V \sharp}\calH^{n}|_{E_{x,r}}\|_{L^2(V)}^{2} \, d\gamma_{d,n}(V) \le Cr^n. 
\end{displaymath}
\end{enumerate}
Then $E$ has BPLG. In particular, $E$ is $n$-UR.
\end{thm}
We now discuss the context and history of the result. The Besicovitch--Federer projection theorem is a characterisation of rectifiability in terms of projections. Since
the introduction of uniform rectifiability, it has been a natural question to find a quantitative analog of the Besicovitch--Federer result for UR sets.
However, an unpublished example of T. Hrycak shows that uniform rectifiability does not imply quantitatively large projections, at least in the obvious sense: given $\epsilon > 0$, Hrycak's construction produces a UR-set $E \subset \R^{2}$, with constants independent of $\epsilon$, such that $\calH^1(\pi_{L}(E)) \leq \epsilon$ for every line $L \in G(2,1)$.

The slightly stronger condition BPLG, however, does imply quantitatively large projections (this is well--known but also follows from Proposition \ref{LInftyProp}), and so characterising BPLG in terms of projections
seems to be a more natural question to ask.
Perhaps the most obvious candidate for such a characterisation is through the \emph{big projections in plenty of directions} (BPPD) assumption.
An $n$-ADR set $E \subset \R^{d}$ has BPPD, if there exists a constant $\delta > 0$ with the following property: for every $x \in E$ and $r \in (0, \diam(E)]$, there is an $n$-plane $V_{x,r} \in G(d,n)$ such that
\begin{equation}\label{BPPD}
\calH^n( \pi_{V}( E \cap B(x,r) )) \ge \delta r^n
\end{equation}
for every $V \in B(V_{x,r}, \delta)$.
Indeed, G. David and S. Semmes ask in \cite{DS1, DS2} whether BPLG is equivalent to BPPD. This remains open to date. It is clear that BPLG implies BPPD. In the converse direction, the quantitative Besicovitch projection theorem of T. Tao \cite{Ta} yields some structural information about BPPD sets, but the conclusions are weaker than BPLG.

It is also known, see \cite{DS2}, that BPLG is characterised by a combination of BPPD and an extra hypothesis called the \emph{the weak geometric lemma} -- an additional regularity assumption not connected with projections. In contrast, our result is the first to characterise BPLG using projections, and projections only. 

Let us briefly see how the BPPD hypothesis is connected with our assumptions. To this end, suppose that a set $E' \subset E \cap B(x,r)$ satisfies $\calH^{n}(E') \gtrsim r^n$, $\pi_{V \sharp}\calH^{n}|_{E'} \in L^2(V)$
and $\|\pi_{V \sharp}\calH^{n}|_{E'}\|_{L^2(V)}^{2} \lesssim r^n$ for some $V \in G(d,n)$. Then,
\begin{displaymath}
r^{2n} \lesssim \|\pi_{V \sharp}\calH^{n}|_{E'}\|_{L^1(V)}^{2} \le \calH^n(\pi_{V}(E')) \cdot \|\pi_{V \sharp}\calH^{n}|_{E'}\|_{L^2(V)}^{2} \lesssim \calH^n(\pi_{V}(E \cap B(x,r))) \cdot r^{n},
\end{displaymath}
which implies \eqref{BPPD} for this particular $V$. Therefore, for those $V$ in Theorem \ref{maingen} such that $\|\pi_{V\sharp}\mathcal{H}^{n}|_{E_{x,r}}\|_{L^2(V)}^{2} \lesssim r^n$, our hypothesis is strictly stronger than \eqref{BPPD}; on the other hand, our "averaged" hypothesis is more relaxed than the uniform requirement of BPPD. 

\subsection{Outline of the proof}
After a suitable translation, scaling and rotation, the proof of Theorem \ref{maingen} reduces to verifying the following statement:
\begin{thm}\label{main} Let $E_{0} \subset \R^d$ be an $n$-ADR set, and assume that $E_1 \subset E_{0} \cap B(0,1)$ is an $\calH^{n}$-measurable subset satisfying the following two properties:
\begin{itemize}
\item[(i)] $\calH^{n}(E_{1}) \ge \kappa > 0$, and
\item[(ii)]
it holds that $\pi_{V \sharp}\calH^{n}|_{E_1} \in L^2(V)$ for $\gamma_{d,n}$-a.e. $V \in B(\R^n, \kappa)$, and 
\begin{displaymath}
\int_{B(\R^n, \kappa)} \|\pi_{V \sharp}\calH^{n}|_{E_1}\|_{L^2(V)}^{2} \, d\gamma_{d,n}(V)  \le C < \infty.
\end{displaymath}
\end{itemize}
Then, there exists a Lipschitz function $A\colon \R^n \to \R^{d-n}$ such that Lip$(A) \lesssim_{\kappa, C} 1$ and
the Lipschitz graph
\begin{displaymath}
\Gamma = \{(x, A(x))\colon x \in \R^n\}
\end{displaymath}
satisfies
\begin{displaymath}
 \quad \calH^{n}(E_{1} \cap \Gamma) \gtrsim_{\kappa,C} 1. 
\end{displaymath}
\end{thm}
The proof divides into one main lemma and one main proposition. 
\begin{definition}[Cones]\label{cones}
For $x \in \R^d$, $V \in G(d,k)$ and $\alpha \in (0,1)$, we set
\begin{displaymath}
X(x,V,\alpha) = \{y \in \R^{d} : |\pi_{V^{\perp}}(x - y)| \leq \alpha|x - y|\}.
\end{displaymath}
 Given a cone $X(x,V,\alpha)$ and two radii $0 < r < R < \infty$, we write
\begin{displaymath} X(x,V,\alpha,R,r) := X(x,V,\alpha) \cap [B(x,R) \setminus U(x,r)], \end{displaymath}
where $B(x,R)$ and $U(x,r)$ are, respectively, the closed and open balls of radii $R > 0$ and $r > 0$ centred at $x$. Note that $X(x,V,\alpha,R,r)$ is a closed set for $0 < r < R < \infty$. Finally, writing $\R^{d} = \R^{n} \times \R^{d - n}$, we use the shorthand notation $X(x,\alpha) := X(x,\R^{d - n},\alpha)$ and $X(x,\alpha,R,r) := X(x,\R^{d - n},\alpha,R,r)$. 
\end{definition}
Our main lemma reads as follows:
\begin{lemma}\label{mainLemma} Assume that $E_0$ and $E_1$ satisfy the hypotheses of Theorem \ref{main}. Then, there exist numbers $M_{0} = M_{0}(\kappa, C) \in \N$,
$\theta_0 = \theta_0(\kappa, d) > 0$ and an $\calH^{n}$-measurable subset $E_{2} \subset E_{1}$ with the following properties: $\calH^{n}(E_{2}) \sim_{\kappa,C} 1$, and if $x \in E_{2}$, then
\begin{displaymath} \#\{j \in \Z \colon\, X(x,\theta_0,2^{-j},2^{-j - 1}) \cap E_{2} \neq \emptyset\} \le M.
 \end{displaymath}
\end{lemma}

Before stating the main proposition, we need another definition:

\begin{definition}\label{thetaM} A set $E \subset \R^{d}$ satisfies the $n$-dimensional $(\theta,M)$-property, if 
\begin{displaymath} \#\{j \in \Z : X(x,\theta,2^{-j},2^{-j - 1}) \cap E\} \leq M \end{displaymath}
for all $x \in E$.
\end{definition} 

\begin{remark}\label{trivialRemark} Observe that if $E$ satisfies the $(\theta,M)$-property with $M = 0$, then $E$ is entirely contained in a Lipschitz graph $\Gamma$ with Lipschitz constant $\leq 1/\theta$. 
Indeed, there holds that
\begin{displaymath}
|\pi_{\R^n}(x-y)| \ge \theta |x-y| , \qquad x,y \in E.
\end{displaymath}
In particular, the restriction $\pi_{\R^{n}}|_{E}$ of the orthogonal projection $\pi_{\R^{n}}$ to $E$ is one-to-one, and one can define a $(1/\theta)$-Lipschitz inverse $f \colon \pi_{\R^{n}}(E) \to E$. By Kirszbraun's theorem, see Theorem 2.10.43 in \cite{Fe}, this can be extended to a Lipschitz mapping $\tilde{f} \colon \R^{n} \to \R^{d}$ with Lip$(\tilde f) \le 1/\theta$,
and then
\begin{displaymath}
E \subset \{(y, A(y))\colon y \in \R^n\},
\end{displaymath}
where $A\colon \R^n \to \R^{d-n}$ is defined by $A = \pi_{\R^{d-n}} \circ \tilde f$. This is essentially the argument from Lemma 15.13 in Mattila's book \cite{Mat}. 
\end{remark}

In the language of Definition \ref{thetaM}, Lemma \ref{mainLemma} claims that $E_{2}$ satisfies the $n$-dimensional $(\theta_{0},M_{0})$-property for some $\theta_{0},M_{0}$ depending only on $\kappa$ and $C$. Here is the main proposition:
\begin{proposition}\label{mainProposition} Assume that $E_{0}$ is $n$-ADR, and assume that $E_{2} \subset E_{0} \cap B(0,1)$ is an $\calH^{n}$-measurable subset with $\calH^{n}(E_{2}) \sim_{\kappa,C} 1$ and satisfying the $n$-dimensional $(\theta,M)$-property for some $\theta > 0$ and $M \geq 0$. Then, there is an $\calH^{n}$-measurable subset $E_{3} \subset E_{2}$ with $\calH^{n}(E_{3}) \sim_{\kappa,C,M, \theta} 1$ and satisfying the $(\theta/b,0)$-property. Here $b \geq 1$ is a constant depending only on $d$.
\end{proposition}

Taking Lemma \ref{mainLemma} and Proposition \ref{mainProposition} for granted, it is straightforward to complete the proof of Theorem \ref{main}:

\begin{proof}[Proof of Theorem \ref{main}] Use Lemma \ref{mainLemma} to find $\theta_{0}$ and $M_{0}$, and the set $E_{2} \subset E_{1}$ satisfying the $n$-dimensional $(\theta_{0},M_{0})$-property. Then, use Proposition \ref{mainProposition} to find $E_{3} \subset E_{2}$ with $\calH^{n}(E_{3}) \sim_{\kappa,C} 1$ and satisfying the $n$-dimensional $(\theta_{0}/b,0)$-property. Now, $E_{3}$ is contained in a Lipschitz graph $\Gamma$ with Lipschitz constant $\leq b/\theta_{0}$ by Remark \ref{trivialRemark}, and
\begin{displaymath} \calH^{n}(E_{1} \cap \Gamma) \geq \calH^{n}(E_{3}) \gtrsim_{\kappa,C} 1. \end{displaymath}
This completes the proof.
\end{proof}

\subsection*{Acknowledgements} We thank the referee for various suggestions, which helped to improve the readability of the paper.

\section{Proof of the main lemma}

We start by proving an easy but very useful auxiliary lemma:

\begin{lemma}\label{auxLemma} Let $E_{0}$ be an $n$-ADR set with $\mathcal{H}^n(E_0) \geq c > 0$, let $E_{1} \subset E_{0} \cap B(0,1)$ be an $\calH^{n}$-measurable subset, and let 
\begin{displaymath} E_{1,\epsilon} := \{x \in E_{1} : \calH^{n}(E_{1} \cap B(x,r_{x})) \leq \epsilon r_{x}^n \text{ for some radius } 0 < r_{x} \leq 1\}. \end{displaymath}
Then $\calH^{n}(E_{1,\epsilon}) \lesssim \epsilon$ with the bound depending only on $c$ and the ADR constant of $E_0$.
\end{lemma}

\begin{proof} The set $E_{1,\epsilon}$ is covered by the balls $B(x,r_{x}/5)$, $x \in E_{1,\epsilon}$, so the $5r$-covering lemma can be used to extract a disjoint subcollection $\{B(x_{i},r_{x_{i}}/5)\}_{i \in \N}$ with the property that the balls $\{B(x_{i},r_{x_{i}})\}_{i \in \N}$ cover $E_{1,\epsilon}$. Let $C_0 > 5$ be so large that $r_{x_i}/C_0 \le \diam(E_0)$ (since $r_{x_{i}} \leq 1$ uniformly, $C_{0}$ depends only on $c$ and the ADR constants of $E_{0}$). Now, we have that
\begin{align*} \calH^{n}(E_{1,\epsilon}) & \leq \sum_{i \in \N} \calH^{n}(E_{1} \cap B(x_{i},r_{x_{i}})) \leq \epsilon \sum_{i \in \N} r_{x_{i}}^n\\
& \lesssim \epsilon \sum_{i \in \N} \calH^{n}(E_{0} \cap B(x_{i},r_{x_{i}}/C_0))\\
& \le \epsilon \cdot \calH^{n}(E_0 \cap B(0,2)) \lesssim \epsilon, \end{align*}
as claimed. \end{proof}
Of course, the lemma cannot be used to conclude that the set $E_{1} \setminus E_{1,\epsilon}$ is $n$-ADR, but it is still somewhat more regular than $E_{1}$, and and this will be useful in the following proofs. 

We also need another technical lemma:
\begin{lemma}\label{GrassmanLemma}
Fix $\upsilon, \delta > 0$ and let $W \in G(d,n)$. Assume that $z \in \R^d$ satisfies $\delta/|z| < \min(\delta_0, \upsilon/2)$, where $\delta_{0}$ is a small constant depending only on $d$, and $|\pi_{W}z| \le \alpha_0|z|$ for a small enough $\alpha_0 = \alpha_0(n, \upsilon) > 0$.
Define $B_z = \{V \in G(d,n)\colon\, |\pi_Vz| \le \delta\}$.
Then,
\begin{displaymath}
A(z) := \gamma_{d,n}(B_z \cap B_{G(d,n)}(W, \upsilon)) \gtrsim_{\upsilon} \Big(\frac{\delta}{|z|}\Big)^{n}.
\end{displaymath}
\end{lemma}
We postpone the proof to Appendix \ref{A1}.

\begin{proof}[Proof of Lemma \ref{mainLemma}] In what follows, the ADR constants of $E_{0}$ and the constants $\kappa$ and $C$ from the statement of Lemma \ref{mainLemma} will be treated as "fixed" in the sense that "$\lesssim_{\kappa,C}$" is abbreviated to "$\lesssim$". We begin by applying Lemma \ref{auxLemma} twice. First, with $\epsilon \sim \calH^{n}(E_{1})$, we remove $E_{1,\epsilon}$ from $E_{1}$: thus, for a suitable $\epsilon \sim 1$, the set $E' := E_{1} \setminus E_{1,\epsilon}$ satisfies $\calH^{n}(E') \sim 1$ and has the property that if $x \in E'$ and $0 < r \leq 1$, then
\begin{equation}\label{form2} \calH^{n}(E_{1} \cap B(x,r)) \sim r^n. \end{equation}
Then, we apply the  lemma again to $E'$, this time with $\epsilon' \sim \calH^{n}(E')$, to the effect that the set $E := E' \setminus E'_{\epsilon'}$ still satisfies $\calH^{n}(E) \sim 1$, and if $x \in E$, $0 < r \leq 1$, then
\begin{equation}\label{form7} \calH^{n}(E' \cap B(x,r)) \sim r^n. \end{equation}

Let $\theta_0 = \alpha_0/2$, where $\alpha_0 > 0$ appears in Lemma \ref{GrassmanLemma} with $\upsilon = \kappa$.
Given $M > 0$ let $E^{M}$ consist of those $x \in E$ for which there are at least $M$ scales $2^{-j}$, $j \in \Z$, such that
\begin{equation}\label{form1} X(x,\theta_0,2^{-j},2^{-j - 1}) \cap E \neq \emptyset. \end{equation} 
In symbols,
\begin{displaymath}
E^M = \{x \in E\colon\,  \#\{j \in \Z \colon\, X(x,\theta_0,2^{-j},2^{-j - 1}) \cap E \neq \emptyset\} \ge M\}.
\end{displaymath}
In a moment, we will show that $\calH^{n}(E^{M}) \lesssim 1/M$; this completes the proof, because then, for a large enough $M$, the set $E_{2} = E \setminus E^{M}$ will be the set we were looking for.

We start with a preliminary reduction. Let $C_0 \geq 1$ be a large constant depending only on $\kappa$, to be specified later. Observe that for every $x \in E^{M}$, there is a constant $\delta_{x} > 0$ such that there are at least $M$ scales $2^{-j} \geq C_0\delta_{x}$ satisfying \eqref{form1}. If $E^{M,\delta} = \{x \in E^{M} : \delta_{x} \geq \delta\}$, we can choose $\delta > 0$ so small that $\calH^{n}(E^{M,\delta}) \geq \calH^{n}(E^{M})/2$. In particular, it suffices to show that $\calH^{n}(E^{M,\delta}) \lesssim 1/M$, where the implicit constant does not depend on $\delta$. This is what we will do, but in order to avoid obscuring the notation any further, we assume that $E^{M} = E^{M,\delta}$; note that \eqref{form2} and \eqref{form7} are obviously unaffected by the passage from $E^{M}$ to $E^{M,\delta}$.

With $\delta$ as in the previous paragraph, let $F^{M}$ and $F'$ be maximal $\delta$-separated sets inside $E^{M}$ and $E'$, respectively; since $E^{M} \subset E'$, we can also arrange so that $F^{M} \subset F'$. We wish to find lower and upper bounds on the amount of triples $(x,y,V)$, where $x,y \in F'$, $V \in B(\R^n, \kappa)$, and 
\begin{displaymath} |\pi_{V}(x-y)| \leq \delta. \end{displaymath}
The idea is that assumption (ii) of Theorem \ref{main} will give us an upper bound for such triples, whereas a lower bound can be obtained, via \eqref{form1}, by choosing $x \in F^{M} \subset F'$ and $y \in F'$. 

We start with the lower bound. Note that 
\begin{equation}\label{form3} \# F^{M} \gtrsim \calH^{n}(E^{M})\delta^{-n}, \end{equation}
because $E^{M}$ is covered by the balls $B(x,2\delta)$, $x \in F^{M}$, and the AD-regularity of $E_{0} \supset E^{M}$ implies that $\calH^{n}(B(x,2\delta) \cap E^{M}) \lesssim \delta^n$. 
Fix $x \in F^{M}$ for the moment. Now, let $2^{-j} \geq C_0\delta$ be one of the scales such that \eqref{form1} holds, and choose a point 
\begin{displaymath} y_{j} \in X(x,\theta_0,2^{-j},2^{-j - 1}) \cap E. \end{displaymath}
First, choose $c_0 = c_0(\theta_0) > 0$ so small that we have
\begin{displaymath} B(y_{j},c_02^{-j}) \subset X(x,2\theta_0, 2^{-j + 1},2^{-j - 2}). \end{displaymath}
This is depicted in Figure \ref{fig4}.
\begin{figure}
\begin{center}
\includegraphics[scale = 0.4]{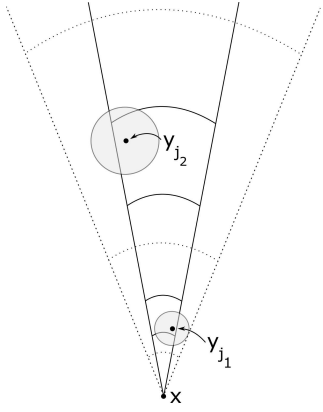}
\caption{Finding the balls $B(y_{j},c2^{-j}) \subset X(x,2\theta_0,2^{-j + 1},2^{-j - 2})$.}\label{fig4}
\end{center}
\end{figure}
Since $y_{j} \in E$, we infer from \eqref{form7} that
\begin{displaymath} \calH^{n}(E' \cap B(y_{j},c_02^{-j-1})) \gtrsim 2^{-jn}. \end{displaymath}
Choosing $C_0 = C_0(\kappa)$ so large that $C_0 \ge 2/c_0$ we have that
\begin{displaymath}
E' \cap B(y_{j},c_02^{-j-1}) \subset \bigcup_{w \in F' \cap B(y_j, c_02^{-j})} B(w,2\delta).
\end{displaymath}
To see this, notice that if $u \in E' \cap B(y_{j},c_02^{-j-1})$,
there is $w \in F'$ so that $|w-u| \le \delta \le 2^{-j}C_0^{-1} \le c_02^{-j-1}$. We now conclude that
\begin{equation}\label{form5} \# (F' \cap B(y_{j}, c_02^{-j})) \gtrsim \delta^{-n}2^{-jn}, \end{equation}
again by the AD-regularity of $E_{0}$. 
In order to use Lemma \ref{GrassmanLemma} we enlarge $C_0 = C_0(\kappa)$ so that for $y \in B(y_{j},c_02^{-j}) \subset X(x,2\theta_0, 2^{-j + 1},2^{-j - 2})$ we have that
\begin{displaymath}
|x-y| \ge 2^{-j-2} \ge \frac{1}{4}C_0\delta \ge \frac{\delta}{\min(\delta_0, \kappa/2)}.
\end{displaymath} 
Hence, Lemma \ref{GrassmanLemma} applies to $z = x - y$ (since $|\pi_{\R^{n}}(z)| \leq 2\theta_{0}|z| = \alpha_{0}|z|$), and so
\begin{displaymath}
\sum_{y \in F' \cap B(y_j, c_02^{-j})} \gamma_{d,n}( B_{x-y} \cap B(\R^n, \kappa)) \gtrsim \delta^{-n}2^{-jn} \Big( \frac{\delta}{2^{-j}} \Big)^{n} = 1,
\end{displaymath}
where $B_{x - y} = B_{z}$ was defined in Lemma \ref{GrassmanLemma}. Then, varying the scale $2^{-j}$, hence the point $y_{j}$, we find $\gtrsim M$ disjoint balls $B(y_j, c_02^{-j})$ inside $X(x, 2\theta_0)$ like above, and therefore
\begin{displaymath}
\sum_{y \in F' \cap X(x, 2\theta_{0})} \gamma_{d,n}( B_{x-y} \cap B(\R^n, \kappa)) \gtrsim M.
\end{displaymath}
This finally yields that
\begin{align*}
\sum_{x, y \in F'} \gamma_{d,n}( B_{x-y} \cap B(\R^n, \kappa)) &\ge \sum_{x \in F^M} \sum_{y \in F' \cap X(x, 2\theta_{0})} \gamma_{d,n}( B_{x-y} \cap B(\R^n, \kappa)) \\
&\gtrsim M \cdot \calH^{n}(E^{M})\delta^{-n}.
\end{align*}

We then go for the upper bound. For $V \in G(d,n)$ define $f_V\colon V \to \R$ by setting
\begin{displaymath}
f_V(z) = \sum_{x \in F'} 1_{B(\pi_Vx, \delta)}(z).
\end{displaymath}
Notice that
\begin{align*}
\int_V f_V(z)^2\,d\calH^n(z) &= \sum_{x,y \in F'} \int_V 1_{B(\pi_Vx, \delta) \cap B(\pi_Vy, \delta)}(z)\,d\calH^n(z) \\
&\gtrsim \delta^n \#\{(x,y) \in F' \times F'\colon\, |\pi_V(x-y)| \le \delta\}. 
\end{align*}
Therefore, we now have that
\begin{align*}
 \calH^{n}(E^{M}) M &\lesssim \delta^n \sum_{x, y \in F'} \gamma_{d,n}( B_{x-y} \cap B(\R^n, \kappa)) \\
 &= \delta^n \int_{B(\R^n, \kappa)} \#\{(x,y) \in F' \times F'\colon\, |\pi_V(x-y)| \le \delta\} \, d\gamma_{d,n}(V) \\
 &\lesssim \int_{B(\R^n, \kappa)}\int_V f_V(z)^2\,d\calH^n(z)\,d\gamma_{d,n}(V).
\end{align*}
Recall that if $x \in F' \subset E'$, then $\calH^{n}(E_{1} \cap B(x,\delta)) \sim \delta^n$ by \eqref{form2}. 
Using this we estimate $f_V$ pointwise:
\begin{align*}
f_V(z) &\le \delta^{-n} \mathop{\sum_{x \in F'}}_{|\pi_Vx-z| \le \delta} \calH^n(E_1 \cap B(x,\delta)) \\
&\lesssim \delta^{-n} \calH^n(E_1 \cap \pi_V^{-1} (B_V(z, 2\delta))) \lesssim M_V(\pi_{V \sharp}\calH^{n}|_{E_1})(z).
\end{align*}
Here $M_{V}$ is the Hardy-Littlewood maximal function on $V$. The operator $M_{V}$ is bounded on $L^{2}(V)$, see Theorem 1 on p. 13 in Stein's book \cite{St}. Therefore, we can conclude the proof as follows, using assumption (ii) from Theorem \ref{main}:
\begin{displaymath}
\calH^{n}(E^{M}) M \lesssim \int_{B(\R^n, \kappa)} \|\pi_{V \sharp}\calH^{n}|_{E_1}\|_{L^2(V)}^{2} \,d\gamma_{d,n}(V) \lesssim 1.
\end{displaymath}
\end{proof}

\section{Proof of the main proposition}
Recall the notation for general cones from Definition \ref{cones}. When $V = \spa(w) \in G(d,1)$, $w \in \Se^{d - 1}$, we introduce shorthand notation for one-dimensional one-sided cones, namely
\begin{displaymath} X^{+}(x,w,\alpha) := X(x,\spa(w),\alpha) \cap \{y \in \R^d \colon\, (y-x) \cdot w \geq 0\}. \end{displaymath}
The restricted version $X^{+}(x,w,\alpha,R,r)$, $0 < r < R < \infty$, is defined in the obvious way. We start with a lemma, which states that cones of arbitrary co-dimension with a fixed aperture can be covered by a bounded number of one-dimensional one-sided cones -- even ones with a slightly smaller aperture:
\begin{lemma}\label{reductionLemma} Fix $\alpha \in (0,1)$, $s \in (0,1]$ and $V \in G(d,k)$. Then, there exist vectors $w_{1},\ldots,w_{m} \in \Se^{d - 1}$, $m \lesssim (\alpha s)^{1 - d}$, such that
\begin{equation}\label{eq:1} X(0,V,\alpha) \subset \bigcup_{j = 1}^{m} X^{+}(0,w_{j},\alpha s) \subset \bigcup_{j = 1}^{m} X^{+}(0,w_{j},\alpha) \subset X(0,V,b\alpha). \end{equation}
Here $b \geq 1$ is a constant depending only on $d$.
\end{lemma}

\begin{proof} It suffices to find $w_{1},\ldots,w_{m} \in \Se^{d - 1}$ such that \eqref{eq:1} holds with all the cones intersected with $\Se^{d-1}$. Find an $(\alpha s)$-net $w_{1},\ldots,w_{m} \subset X(0,V,\alpha) \cap \Se^{d - 1}$. Then $m \lesssim (\alpha s)^{1 - d}$. Now, write $W_{j} := w_{j}^{\perp}$, fix  $y \in X(0,V,\alpha) \cap \Se^{d - 1}$, and pick $w_{j}$ such that $|y - w_{j}| \leq \alpha s$. Then
$y \cdot w_{j} \geq 0$ and
\begin{displaymath} |\pi_{W_{j}}(y)| = |\pi_{W_{j}}(y - w_{j})| \leq \alpha s, \end{displaymath}
which proves that $y \in X^{+}(0,w_{j},\alpha s) \cap \Se^{d - 1}$ and hence the first inclusion of \eqref{eq:1}. To prove the second inclusion, fix $1 \leq j \leq m$ and $z \in X^{+}(0,w_{j},\alpha) \cap \Se^{d - 1}$. Observe that $|z-w_{j}| \leq b'\alpha$ for some constant $b' = b'(d)$, whence
\begin{displaymath} |\pi_{V^{\perp}}(z)| \leq |\pi_{V^{\perp}}(z - w_{j})| + |\pi_{V^{\perp}}(w_{j})| \leq |z - w_{j}| + \alpha \leq (b' + 1)\alpha. \end{displaymath}
This proves the lemma with $b = b' + 1$. \end{proof}

Lemma \ref{reductionLemma} allows us to reduce the proof of the main Proposition \ref{mainProposition} to the "co-dimension $1$" case. The relevant (one-dimensional, one-sided) cones in this co-dimension are directed, which is quite useful in the proof.

\begin{proposition}[Main proposition in co-dimension $1$]\label{mainProp2}
Let $w \in \Se^{d-1}$. Assume that an $\calH^{n}$-measurable set $F \subset \R^{d}$ satisfies the following conditions: 
\begin{itemize}
\item[(a)] $\calH^{n}(F) =: \tau > 0$, and $F \subset E_{0} \cap B(0,1)$ for some $n$-ADR set $E_{0}$,
\item[(b)] for some $\alpha > 0$, $M \in \N$, and for every point $x \in F$ there holds that
\begin{displaymath} \#\{j \in \Z\colon\,X^+(x, w, \alpha,2^{-j},2^{-j - 1}) \cap F \neq \emptyset\} \le M. \end{displaymath}
\end{itemize}
We abbreviate (a) and (b) by saying that that $F$ satisfies the $(\alpha,M,w)$-property. Then, if $M \geq 1$ and $\alpha > 0$ is small enough (depending on only on $d$), there exists a compact set $K \subset F$ with $\calH^{n}(K) \sim_{\tau,\alpha} 1$, which satisfies the $(\alpha/2,M - 1,w)$-property.
\end{proposition}

The general idea of the argument is to write down an explicit algorithm, which refines $F$ by deleting some points in several stages, but all the time keeps track that not too much is wasted. When the algorithm eventually stops, it will output the desired set $K$. Before giving the details, let us see how the general version of the main proposition follows from Proposition \ref{mainProp2}.
\begin{proof}[Proof of Proposition \ref{mainProposition}] Assume that $E_{2} \subset E_{0}$ satisfies the $n$-dimensional $(\theta_{0},M_{0})$-property. Then, with $V = (\R^{n})^{\perp}$, 
\begin{displaymath} \# \{j \in \Z : X(x,V,\theta_{0}, 2^{-j}, 2^{-j-1}) \cap E_{2} \neq \emptyset\} \leq M_{0} \end{displaymath}
for all $x \in E_2$. Now, use Lemma \ref{reductionLemma} with $\alpha = \theta_{0}/b$ and $s = 2^{-M_{0}}$ to find vectors $w_{1},\ldots,w_{m} \in \Se^{1}$, $m \lesssim (\theta_{0}/(2^{M_{0}}b))^{1 - d}$, such that
\begin{displaymath} X(0,V,\theta_{0}/b) \subset \bigcup_{j = 1}^{m} X^{+}(0,w_{j},\theta_{0}/(2^{M_{0}}b)) \subset \bigcup_{j = 1}^{m} X^{+}(0,w_{j},\theta_{0}/b) \subset X(0,V,\theta_{0}). \end{displaymath}
It follows from translation invariance that
\begin{equation}\label{eq:3} X(x,V,\theta_{0}/b,2^{-j},2^{-j - 1}) \subset \bigcup_{j = 1}^{m} X^{+}(x,w_{j},\theta_{0}/(2^{M_{0}}b),2^{-j},2^{-j - 1}) \end{equation}
for all $x \in \R^{d}$ and $j \in \Z$. Further, $E_{2}$ satisfies the $(\theta_{0}/b,M_{0},w_{j})$-property from Proposition \ref{mainProp2} for all $1 \leq j \leq m$. Thus, iterating that proposition $\leq M_{0}$ times for each $1 \leq j \leq m$, hence $\leq mM_{0}$ times altogether, one finds a set $E_{3} \subset E_{2}$ satisfying the $(\theta_{0}/(2^{M_{0}}b),0,w_{j})$-property for all $1 \leq j \leq m$. It follows from \eqref{eq:3} that $E_{3}$ satisfies the $n$-dimensional $(\theta_{0}/b,0)$-property, and Proposition \ref{mainProposition} is proved. \end{proof}

\begin{proof}[Proof of Proposition \ref{mainProp2}]
We assume that $w = e_d = (0, 0, \ldots, 0, 1)$.
Before starting to describe the algorithm to find $K \subset F$, we make two easy reductions: first, without loss of generality, we may assume that if $x \in F$ and 
\begin{displaymath}
X^+(x, w, \alpha,2^{-j},2^{-j - 1}) \cap F \neq \emptyset
\end{displaymath}
then $2^{-j} \geq \delta$ for some small constant $\delta > 0$. Simply, for every $x \in F$, there is some $\delta_{x} > 0$ with this property, and then we can take $\delta > 0$ so small that $\calH^{n}(F') \geq \calH^{n}(F)/2$, where $F' := F \setminus \{x \in F : \delta_{x} < \delta\}$. After this, we would proceed with the proof as below, only replacing $F$ by $F'$. Second, we may assume that $F$ is compact; otherwise we can always find a compact subset of $F$ (or $F'$) with almost the same $\calH^{n}$-measure, and then we can find $K$ inside this subset as below.

We now begin to describe the algorithm. The following points (I)--(IV) summarise the key features.
\begin{itemize} 
\item[(I)] There will be a sequence of compact sets $F = F^{0} \supset F^{1} \supset F^{2} \supset \cdots$, where $F^{k + 1}$ is obtained from $F^{k}$ by deleting a certain open set $D^{k}$.
\item[(II)] Thus, there will also be a sequence of \emph{deleted sets} $D^{k} \subset F^{k}$, $k \in \{0,1,\ldots\}$. 
\item[(III)] There will be a sequence of \emph{saved sets} $S^{k} \subset F^{k} \subset F$, $k = \{0,1,\ldots\}$, which are disjoint from each other and all the deleted sets $D^{i}$, $i \geq k$,\footnote{The sets $S^{k}$ are also disjoint from the deleted sets $D^{i}$ with $i < k$, as $S^{k} \subset F^{k} = F \setminus \bigcup_{i < k} D^{i}$.} satisfy
\begin{displaymath} \calH^{n}(S^{k}) \gtrsim \max\{\calH^{n}(D^{k}),\delta^n\}, \end{displaymath}
and have the property that if 
\begin{displaymath} x \in \bigcup_{i \leq k} S^{i}, \end{displaymath}
then there are at most $M - 1$ scales $2^{-j}$ such that
\begin{displaymath} X^+(x,w, \alpha/2,2^{-j},2^{-j - 1}) \cap \bigcup_{i \leq k} S^{i} \neq \emptyset. \end{displaymath} 
\item[(IV)] We describe the structure of the saved sets. Let $F^{k,M}$ be the set of points in $F^{k}$ such that there are exactly $M$ scales $2^{-j} \geq \delta$ such that 
\begin{displaymath} X^+(x,w,\alpha/2,2^{-j},2^{-j - 1}) \cap F^{k} \neq \emptyset. \end{displaymath} 
A point $x \in F^{k}$ is then called \emph{$k$-bad}, if $x \in F^{k,M}$, and furthermore
\begin{displaymath}
\calH^{n}(B(x,r) \cap F^{k,M}) \geq \epsilon r^n
\end{displaymath}
for all radii $0 < r \leq 1$, where $\epsilon \sim \calH^{n}(F)$ is a constant to be specified in Stopping condition \ref{secondStopCondition} below. Using the compactness of $F^{k}$ and the uniform lower bound for the numbers $2^{-j}$, it is easy to verify that the set of $k$-bad points is compact. Thus, if there are any $k$-bad points to begin with, there exists a (possibly non-unique) $k$-bad point $x_{k}$ with the smallest last coordinate $x_{k}^{d}$. With such a choice of $x_{k}$, the saved set $S^{k}$ will be defined as $B(x_{k},r_{k}) \cap F^{k,M}$ for some suitable radius $r_{k} \gtrsim \delta$. 

Note that if $x$ is $k$-bad and $k \geq 1$, then $x$ is also $(k - 1)$-bad, simply because $F^{k} \subset F^{k - 1}$ and $F^{k,M} \subset F^{k - 1, M}$. This implies, by the definition of $x_{k}$, that the last coordinates of the points $x_{0},x_{1},\ldots,x_{k}$ form a non-decreasing sequence. 

Finally, to every set $S^{k} = B(x_{k},r_{k}) \cap F^{k,M}$ we associate a somewhat larger set $B^{k} := B(x_k,100r_{k}) \cap F^{k}$, which will have the property that if $x \in B^{k}$, then there are at most $M - 1$ scales $2^{-j}$ such that
\begin{displaymath}
X^+(x,w,\alpha/2,2^{-j},2^{-j - 1}) \cap F^{k + 1} \neq \emptyset.
\end{displaymath}
\end{itemize}
There will be two different stopping conditions, which bring the algorithm to a halt and output the desired set $K$. 
\begin{stopCondition}\label{firstStopCondition} Assume that the sets $D^{0},\ldots,D^{k}$ and $S^{0},\ldots,S^{k}$ have been defined, and either 
\begin{displaymath} \sum_{i = 0}^{k} \calH^{n}(D^{i}) \geq \calH^{n}(F)/2 \end{displaymath}
or
\begin{displaymath} \sum_{i = 0}^{k} \calH^{n}(S^{i}) \geq \calH^{n}(F)/2. \end{displaymath}
In both cases, we set
\begin{displaymath} K := \bigcup_{i \leq k} S^{i}. \end{displaymath}
By (III), the set $K$ satisfies the requirements of Proposition \ref{mainProp2}, and the proof is complete.
\end{stopCondition}

\begin{stopCondition}\label{secondStopCondition} Assume that the set $F^{k}$ has been defined, and satisfies $\calH^{n}(F^{k}) \geq \calH^{n}(F)/2$, and that the set of $k$-bad points, as in (IV), is empty. Thus, for every $x \in F^{k,M}$, we have
\begin{displaymath}
\calH^{n}(F^{k,M} \cap B(x,r_{x})) \leq \epsilon r_{x}^n
\end{displaymath}
for some radius $0 < r_{x} \leq 1$. Now choose $\epsilon \sim \calH^{n}(F) \sim 1$ so small that, using Lemma \ref{auxLemma}, we have $\calH^{n}(F^{k,M}) \leq \calH^{n}(F)/4$. We set
\begin{displaymath} K := F^{k} \setminus F^{k,M}. \end{displaymath}
Then, $\calH^{n}(K) \geq \calH^{n}(F^{k}) - \calH^{n}(F^{k,M}) \geq \calH^{n}(F)/4$, and for every $x \in K$ there are at most $M - 1$ scales $2^{-j}$ such that
\begin{displaymath}
X^+(x,w,\alpha/2,2^{-j},2^{-j - 1}) \cap K \neq \emptyset.
\end{displaymath}
Thus, $K$ satisfies the requirements of Proposition \ref{mainProp2}, and the proof is complete. 
\end{stopCondition}

\begin{remark} Notice that, since $\calH^{n}(S^{k}) \gtrsim \delta^n$ for every $k$, the first stopping condition will be reached in $\lesssim \delta^{-n}$ steps (unless the second stopping condition was reached before that). In particular, the algorithm terminates and outputs $K$ after finitely many steps.
\end{remark}

Next, we will explicitly describe how to construct the various sets $F^{k}$, $S^{k}$ and $D^{k}$. Define $S^{-1} = \emptyset$, $D^{-1} = \emptyset$ and $F^{0} := F$. Assume that $k \geq 0$ and the sets $F^{0},\ldots,F^{k}$, $D^{1},\ldots,D^{k - 1}$ and $S^{1},\ldots,S^{k - 1}$ have already been defined, and satisfy the properties listed in (I)--(IV); in particular, also the balls $B^{i}$, $i < k$, have been defined. Assume that the first stopping condition is not satisfied; otherwise the algorithm terminates and the proof is complete. In particular,
\begin{equation}\label{form8}
\calH^{n}(F^{k}) \geq \calH^{n}(F) - \sum_{i < k} \calH^{n}(D^{i}) \geq \calH^{n}(F)/2.
\end{equation}
Next, assume that the second stopping condition is not satisfied; because of \eqref{form8}, this means that the set of $k$-bad points is non-empty, and -- as required by (IV) -- we find one of them, $x_{k}$, with minimal last coordinate. Let $2^{-j_k}$ be one of the $M$ scales such that
\begin{equation}\label{form18} X^+(x_{k},w, \alpha/2,2^{-j_k},2^{-j_k - 1}) \cap F^{k} \neq \emptyset. \end{equation}
Let $r_{k} := c2^{-j_k}$ for a suitable small $c = c(d,\alpha)> 0$ to be specified later, and set
\begin{displaymath} S^{k} := B(x_{k},r_{k}) \cap F^{k,M}, \qquad B^{k} := B(x_{k},100r_{k}) \cap F^{k}, \end{displaymath}
as required by (IV). Then 
\begin{equation}\label{form11}
\calH^{n}(S^{k}) \gtrsim 2^{-j_kn} \geq \delta^n,
\end{equation}
by the definition of $k$-badness. Furthermore, $S^{k}$ is disjoint from all the previous sets $S^{i}$, $i < k$, and even the larger sets $B^{i}$, $i < k$, because $S^{k} \subset F^{k,M}$, but if $x \in B^{i}$, then 
\begin{equation}\label{form19}
X^+(x,w,\alpha/2,2^{-m},2^{-m - 1}) \cap F^{k} \neq \emptyset
\end{equation}
can only hold for $M - 1$ scales $2^{-m}$ by (IV). Next, we define the deleted set $D^{k}$ by
\begin{displaymath} D^{k} := F^{k} \cap \bigcup_{l \in \{-1,0,1\}} \bigcup_{x \in B^{k}} X^{\circ}(x,w,\alpha,2^{-j_k + l}, 2^{-j_k - 1 + l}). \end{displaymath}
Here $X^{\circ}$ stands for the interior of the cone $X^+$ (we want the deleted set to be relatively open in $F^{k}$). Then $D^{k}$ is contained in a single ball of radius $\lesssim 2^{-j_k}$, so $\calH^{n}(D^{k}) \lesssim 2^{-j_kn}$. Combining this with \eqref{form11}, we see that
\begin{displaymath} \calH^{n}(S^{k}) \gtrsim \max\{\calH^{n}(D^{k}), \delta^n\}, \end{displaymath}
as required in (III).

To complete the proof, we still need to show the disjointness of $D_{k}$ from the previous saved sets $S^{i}$, $i < k$, the latter claim in (III) about $\bigcup_{i \leq k} S^{i}$, and the claim about the set $B^{k}$ at the end of (IV). We begin with the last and easiest task. By the definition of $2^{-j_{k}}$ in \eqref{form18}, there exists a point
\begin{displaymath} z_{k} \in X^+(x_{k},w,\alpha/2,2^{-j_k},2^{-j_k - 1}) \cap F^{k}. \end{displaymath}
Now, if the constant $c$ in $r_{k} = c2^{-j_k}$ is chosen small enough (depending on $\alpha$), and $x \in B^{k} \subset B(x_{k},100r_{k})$, one can check that 
\begin{displaymath} z_{k} \in \bigcup_{l \in \{-1,0,1\}} X^+(x,w,\alpha,2^{-j_k + l}, 2^{-j_k - 1 + l}), \end{displaymath}
\begin{figure}[h!]
\begin{center}
\includegraphics[scale = 0.7]{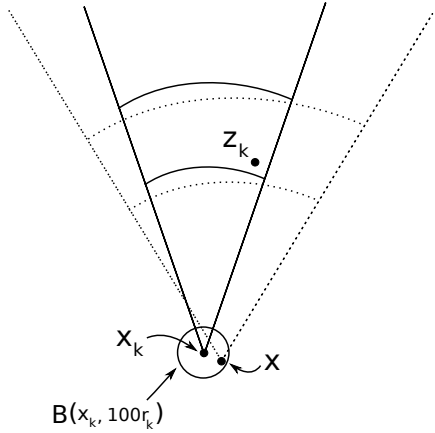}
\caption{Some of the points and regions associated with $D^{k}$.}\label{fig1}
\end{center}
\end{figure}
see Figure \ref{fig1}. In particular, one of the three scales $2^{-j_k + l}$, $l \in \{-1,0,1\}$, is among the at most $M$ scales $2^{-m}$ such that $X^+(x,w,\alpha,2^{-m}, 2^{-m - 1}) \cap F^{k} \neq \emptyset$. Then $D^{k}$ certainly contains all the points in the intersection $X^+(x,w,\alpha/2,2^{-m},2^{-m - 1}) \cap F^{k}$, so
\begin{displaymath} X^+(x,w,\alpha/2,2^{-m},2^{-m - 1}) \cap F^{k + 1} = X^+(x,w,\alpha/2,2^{-m},2^{-m - 1}) \cap (F^{k} \setminus D^{k}) = \emptyset. \end{displaymath}
Thus, there can only remain at most $M - 1$ scales $2^{-m}$ such that
\begin{equation}\label{form20}
X^+(x,w,\alpha/2,2^{-m},2^{-m - 1}) \cap F^{k + 1} \ne \emptyset, \quad x \in B^{k},
\end{equation}
and this is exactly what is claimed at the end of (IV).

Finally, we establish the remaining claims in (III) by proving that $D^{k}$ is disjoint from the saved sets $S^{i}$, $i \leq k$. In fact, this implies that $S^{i} \cap D^{l} = \emptyset$ for all pairs $i,l \leq k$. Indeed, if $i \leq l < k$, we may assume by induction that $S^{i}$ is disjoint from $D^{l}$ (since this is precisely what we are about to prove for $l = k$). Further, $S^{i}$ is disjoint from $D^{l}$ with $l < i$ simply because
\begin{displaymath} S^{i} \subset F^{i} = F \setminus \bigcup_{l < i} D^{l}. \end{displaymath}
From the previous discussion, we conclude that
\begin{equation}\label{form14} \bigcup_{i \leq k} S^{i} \subset F \setminus \bigcup_{l \leq k} D^{l} = F^{k + 1}. \end{equation}
Observe that for every $x \in \bigcup_{i \leq k} S^{i}$ there are at most $M - 1$ scales $2^{-j}$ such that 
\begin{displaymath} X^+(x,w,\alpha/2,2^{-j},2^{-j - 1}) \cap F^{k + 1} \neq \emptyset. \end{displaymath} 
This follows from \eqref{form20} for $x \in S^{k} \subset B^{k}$, and from induction for $x \in S^{i} \subset B^{i}$ for $i < k$ (recalling \eqref{form19} and noting that $F_{k + 1} \subset F_{k}$). Hence, we infer from the inclusion \eqref{form14} that there are also at most $M - 1$ scales $2^{-j}$ such that
\begin{displaymath} X^+(x,w,\alpha/2,2^{-j},2^{-j - 1}) \cap \bigcup_{i \leq k} S^{i} \neq \emptyset \end{displaymath} 
for $x \in \bigcup_{i \leq k} S^{i}$. This is what was claimed at the end of (III). 
\begin{figure}[h!]
\begin{center}
\includegraphics[scale = 0.7]{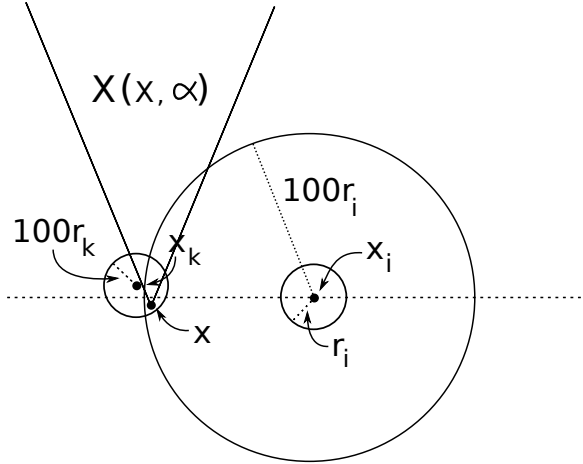}
\caption{The case $100r_{k} \leq r_{i}$.}\label{fig2}
\end{center}
\end{figure}
Now, we fix $i \leq k$, and establish that $D^{k}$ is disjoint from $S^{i}$. If $i = k$, this is immediate from the construction (recall that $S^{k} \subset B(x_{k},r_{k})$, whereas $D^{k}$ lies inside the union of certain annuli, all at distance $\gg r_{k}$ from $x_{k}$). So, we assume that $i < k$. 
There are two cases to consider. First, assume that $100r_{k} \leq r_{i}$ (see Figure \ref{fig2}). In this case, we simply prove that if $x \in B^{k} \subset B(x_{k},100r_{k}) \subset B(x_{k},r_{i})$, then
\begin{equation}\label{form15} X^+(x,w,\alpha) \cap B(x_{i},r_{i}) = \emptyset, \end{equation}
which is clearly a stronger statement than $D^{k} \cap S^{i} = \emptyset$.
Fix $x \in B^{k}$, and recall that $x_{k} \notin B^{i} = B(x_{i},100r_{i}) \cap F^{i}$ for $i < k$ (because $x_{k} \in S_{k}$, and $S_{k}$ is disjoint from $B^{i}$, as remarked below \eqref{form11}). Because $x_{k} \in S^{k} \subset F^{k} \subset F^{i}$, this implies that $x_{k} \notin B(x_{i},100r_{i})$, and hence, by $100r_{k} \leq r_{i}$,
\begin{equation}\label{form12} x \notin B(x_{i},50r_{i}). \end{equation}
Now, recall that the last coordinate of $x_{k}$ is no smaller than the last coordinate of $x_{i}$ by (IV). So, if we write $y = (y^{u})_{u=1}^d$ for a general point $y \in \R^d$, we have
\begin{equation}\label{form13} x^{d} \geq x_{k}^{d} - r_{i} \ge x_i^d - r_i. \end{equation}
It is now easy to check, based on \eqref{form12} and \eqref{form13} that \eqref{form15} holds, if we assume that, say, $\alpha \leq 1/10$.
\begin{figure}[h!]
\begin{center}
\includegraphics[scale = 0.3]{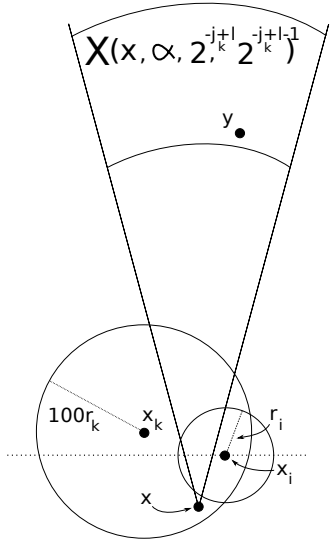}
\caption{The case $100r_{k} > r_{i}$.}\label{fig3}
\end{center}
\end{figure} 

Next, assume that $100r_{k} > r_{i}$ (see Figure \ref{fig3}). Recall that $r_{k} = c2^{-j_k}$, and $D^{k}$ is contained in the union of the annuli $X^{\circ}(x,w,\alpha,2^{-j_k + l}, 2^{-j_k - 1 + l})$, where $x \in B^{k} \subset B(x_{k},100r_{k})$ and $l \in \{-1,0,1\}$. If $x \in B^{k}$ is fixed, then by the same argument that gave \eqref{form13}, we now have
\begin{equation}\label{form16} x^{d} \geq x_{k}^{d} - 100r_{k} \geq x_{i}^{d} - 100c2^{-j_k}. \end{equation} 
If
\begin{displaymath} y  \in X^{\circ}(x,w,\alpha,2^{-j_k + l},2^{-j_k - 1 + l}), \end{displaymath}
then, using \eqref{form16} and choosing $c > 0$ small enough,
\begin{displaymath} y^{d} \geq x^{d} + 2^{-j_k - 10} \geq x_{i}^{d} + 2^{-j_k - 20} \geq x_{i}^{d} + 1000c2^{-j_k} \geq x_{i}^{d} + 10r_{i}. \end{displaymath}
In particular, $y$ cannot lie in $B(x_{i},r_{i}) \supset S^{i}$, and the proof is complete. 
 \end{proof}

\appendix
\section{A measure estimate on the Grassmannian}\label{A1}
This section contains the proof of Lemma \ref{GrassmanLemma}. Let us recall the statement:
\begin{lemma}
Fix $\upsilon, \delta > 0$ and let $W \in G(d,n)$. Assume that $z \in \R^d$ satisfies $\delta/|z| < \min(\delta_0, \upsilon/2)$, where $\delta_{0}$ is a small constant depending only on $d$, and $|\pi_{W}z| \le \alpha_0|z|$ for a small enough $\alpha_0 = \alpha_0(n, \upsilon) > 0$.
Define $B_z = \{V \in G(d,n)\colon\, |\pi_Vz| \le \delta\}$.
Then,
\begin{displaymath}
A(z) := \gamma_{d,n}(B_z \cap B_{G(d,n)}(W, \upsilon)) \gtrsim_{\upsilon} \Big(\frac{\delta}{|z|}\Big)^{n}.
\end{displaymath}
\end{lemma}
\begin{proof}
Let us begin by showing that there exists $V_0 \in B_{G(d,n)}(W, \upsilon/2)$ for which $\pi_{V_0}z = 0$. Let $e_1, \ldots, e_n \in \R^d$
be an orthonormal basis for $W$. Notice that for every $j \in \{1, \ldots, n\}$ we have that
\begin{equation*}
|z \cdot e_j| \le |\pi_W z| \le \alpha_0|z|.
\end{equation*}
The plan is now to form new vectors $u_{1},\ldots,u_{n}$ by perturbing the vectors $e_{1},\ldots,e_{n}$ slightly. Let $u_0 := e_{0} := z/|z|$, $Z_{-1} := \emptyset$ and $\epsilon_0 := \alpha_{0}$.
Assume that $u_0, \ldots, u_k$, $0 \leq k < n$, have already been defined so that they satisfy:
\begin{enumerate}
\item $(u_0, u_1, \ldots, u_k)$ is an orthonormal sequence;
\item If $Z_i = \textup{span}(u_0, u_1, \ldots, u_i)$ for $i \in \{0, \ldots, k\}$ then 
\begin{displaymath}
u_{i} = \frac{\pi_{Z_{i-1}^{\perp}}e_{i}}{|\pi_{Z_{i-1}^{\perp}}e_{i}|}
\end{displaymath}
and
\begin{equation}\label{eq:5}
|\pi_{Z_i}e_{i+1}| \le \epsilon_i.
\end{equation}
\end{enumerate}
As will be apparent in a moment, the numbers $\epsilon_{i}$ will be defined via a simple recurrence relation. Observe that \eqref{eq:5} also gives
\begin{displaymath}
|\pi_{Z_{i}^{\perp}}e_{i+1}| \ge (1-\epsilon_{i}^2)^{1/2}, \qquad i \in \{0, \ldots, k\}.
\end{displaymath}

Now, let 
\begin{displaymath}
u_{k+1} = \frac{\pi_{Z_{k}^{\perp}}e_{k+1}}{|\pi_{Z_{k}^{\perp}}e_{k+1}|}
\end{displaymath}
and $Z_{k+1} = \textup{span}(Z_{k}, u_{k+1})$. If $k = n - 1$, the vectors $\{u_{1},\ldots,u_{n}\}$ have now been defined, and the induction terminates. If $k < n - 1$, notice that
\begin{align*}
|\pi_{Z_{k+1}}e_{k+2}|^2 &=  \sum_{i=0}^{k+1} (e_{k+2} \cdot u_i)^2 \\
& = (e_{k+2} \cdot u_0)^2 +  \sum_{i=1}^{k+1} \Big(e_{k+2} \cdot \frac{\pi_{Z_{i-1}^{\perp}}e_{i}}{|\pi_{Z_{i-1}^{\perp}}e_{i}|} \Big)^2 \\
&=  (e_{k+2} \cdot u_0)^2 +  \sum_{i=1}^{k+1} \Big(e_{k+2} \cdot \frac{\pi_{Z_{i-1}}e_{i}}{|\pi_{Z_{i-1}^{\perp}}e_{i}|} \Big)^2 \\
&\le \alpha_0^2 + \sum_{i=1}^{k+1} \frac{\epsilon_{i-1}^2}{1-\epsilon_{i-1}^2} =: \epsilon_{k + 1}^2.
\end{align*}
Now, properties (1) and (2) have been verified for the vectors $u_{0},\ldots,u_{k + 1}$.

Define $V_0 = \textup{span}(u_1, \ldots, u_n)$. By (1) it follows that $\pi_{V_0}z = 0$. Therefore, it remains to show that
$V_0 \in B_{G(d,n)}(W, \upsilon/2)$ for a small enough $\alpha_0$. First, it is easy to check using the definition of $u_i$ that
\begin{align*}
u_i - e_i &= \frac{-\pi_{Z_{i-1}} e_i}{|\pi_{Z_{i-1}^{\perp}}e_i|} + \Big( \frac{1}{|\pi_{Z_{i-1}^{\perp}}e_i|} - 1\Big)e_i.
\end{align*}
It follows that
\begin{displaymath}
|u_i - e_i| \le \frac{\epsilon_{i-1}}{(1-\epsilon_{i-1}^2)^{1/2}} +  \Big( \frac{1}{(1-\epsilon_{i-1}^2)^{1/2}} - 1\Big) =: r_i.
\end{displaymath}
Choose $\alpha_0$ so small that $\max\{r_i\colon i = 1, \ldots,n\} \le \upsilon/(4n^{1/2})$; this can clearly be done, given that the numbers $\epsilon_{i}$ satisfy the recurrence relation above. It follows that
\begin{displaymath}
\|W-V_0\|_{G(d,n)} = \|\pi_W-\pi_{V_0}\| \le \upsilon/4,
\end{displaymath}
ending the proof of the existence of $V_0$.

Let $F = \{V \in G(d,n)\colon \pi_Vz = 0\}$ and identify $F$ with $G(d-1,n)$ (notice that $F$ is exactly the $n$-planes contained in $z^{\perp} \approx \R^{d-1}$). Now $V_0 \in F$. Let $H$
be a maximal $\delta/|z|$-separated collection of $V \in B_{G(d-1,n)}(V_0, \upsilon/2)$. Then
$B_{G(d-1,n)}(V_0, \upsilon/2) \subset \bigcup\{ B_{G(d-1,n)}(V, 2\delta/|z|)\colon\, V \in H\}$ yielding that
\begin{displaymath}
1 \lesssim_{\upsilon} \#H \cdot \Big(\frac{\delta}{|z|}\Big)^{n(d-1-n)}.
\end{displaymath}
Here we used that $\gamma_{d-1,n}( B_{G(d-1,n)}(V_0, \upsilon/2)) \sim_{\upsilon} 1$ and Proposition 4.1 of \cite{FO} (we also implicitly used $\delta/|z| < \delta_{0}$, which, for small enough $\delta_{0}$, guarantees that Proposition 4.1 of \cite{FO} applies to the balls $B_{G(d - 1,n)}(V,2\delta/|z|)$).
Notice that
\begin{displaymath}
\bigcup_{V \in H} B_{G(d,n)}(V, \delta/(2|z|)) \subset B_z \cap B_{G(d,n)}(W, \upsilon)
\end{displaymath}
by the definition of $B_{z}$, the inclusion $H \subset F$, and the inequalities $\delta/(2|z|) \leq \upsilon/4$, and
\begin{displaymath} \|V - W\|_{G(d,n)} \leq \|V - V_{0}\|_{G(d - 1,n)} + \|V_{0} - W\|_{G(d,n)} \leq \frac{3\upsilon}{4}, \qquad V \in H. \end{displaymath}
Also, the $G(d,n)$-balls in the union are disjoint, so
\begin{displaymath}
A(z) \geq \sum_{V \in H} \gamma_{d,n}(B_{G(d,n)}(V,\delta/(2|z|))) \gtrsim_{\upsilon} \Big(\frac{\delta}{|z|}\Big)^{n(d-n)} \cdot \Big(\frac{\delta}{|z|}\Big)^{-n(d-1-n)} = \Big(\frac{\delta}{|z|}\Big)^n,
\end{displaymath}
using the cardinality estimate for $H$, and Proposition 4.1 of \cite{FO} again.
\end{proof}


\begin{bibdiv}
\begin{biblist}

\bib{AS}{article}{
   author = {Azzam, Jonas},
   author = {Schul, Raanan},
   title = {Hard Sard: quantitative implicit function and extension theorems for Lipschitz maps},
   journal={Geom. Funct. Anal.},
   volume={22},
   date={2012},
   number={5},
   pages={1062--1123},
}

\bib{DS1}{book}{
   author={David, Guy},
   author={Semmes, Stephen},
   title={Analysis of and on uniformly rectifiable sets},
   series={Mathematical Surveys and Monographs},
   volume={38},
   publisher={American Mathematical Society, Providence, RI},
   date={1993},
   pages={xii+356},
}

\bib{DS3}{article}{
   author={David, Guy},
   author={Semmes, Stephen},
   title={Singular integrals and rectifiable sets in ${\bf R}\sp n$: Beyond
   Lipschitz graphs},
    journal={Ast\'erisque},
   number={193},
   date={1991},
   pages={152},
   issn={0303-1179},
}

\bib{DS2}{article}{
   author={David, Guy},
   author={Semmes, Stephen},
   title={Quantitative rectifiability and Lipschitz mappings},
   journal={Trans. Amer. Math. Soc.},
   volume={337},
   date={1993},
   number={2},
   pages={855--889},
   issn={0002-9947},
}

\bib{FO}{article}{
   author ={F\"assler, Katrin},
   author = {Orponen, Tuomas},
   title = {Constancy results for special families of projections},
   journal = {Math. Proc. Cambridge Philos. Soc.}
   volume = {154},
   number = {3},
   year = {2013},
   pages = {549--568},
}

\bib{Fe}{book}{
   author = {Federer, Herbert}
   title = {Geometric Measure Theory}
   publisher = {Springer-Verlag}
   date = {1969}
   }

\bib{Mat}{book}{
   author={Mattila, Pertti},
   title={Geometry of sets and measures in Euclidean spaces: Fractals and rectifiability},
   series={Cambridge Studies in Advanced Mathematics},
   volume={44},
   publisher={Cambridge University Press, Cambridge},
   date={1995},
   pages={xii+343},
   isbn={0-521-46576-1},
   isbn={0-521-65595-1},
   doi={10.1017/CBO9780511623813},
}

\bib{St}{book}{
    author = {Stein, Elias M.}
    title = {Harmonic Analysis: Real-variable Methods, Orthogonality, and Oscillatory Integrals}
    publisher = {Princeton University Press, Princeton, New Jersey}
    date = {1993}
    isbn = {0-691-03216-5}
}
    

\bib{Ta}{article}{
   author = {Tao, Terence},
   title = {A quantitative version of the Besicovitch projection theorem via multiscale analysis},
   journal = {Proc. London Math. Soc.},
   date = {2009},
   volume = {98},
   number = {3},
   pages = {559--584},
}   

\end{biblist}
\end{bibdiv}
\end{document}